\documentclass[11pt]{article}
\usepackage{amsmath,amsthm, amssymb}
\usepackage{graphicx}
\usepackage{fullpage}

\newcommand{\R}{{{\mathbb {R}}}}
\newcommand{\C}{{{\mathbb C}}}

\newcommand{\U}{{\mathbb U}}
\newtheorem{Theorem}{Theorem}
\newtheorem{Definition}{Definition}

\newtheorem {lemma} [Theorem]    {Lemma}
\newtheorem {corollary}  [Theorem]    {Corollary}
\newtheorem {Proposition}[Theorem]    {Proposition}
\newtheorem {proposition}[Theorem]    {Proposition}

\newcommand{\eps}{\epsilon}
\newcommand{\geps}{\Gamma_\epsilon}
\newcommand{\PP}{{ P}}
\newcommand{\EE}{{ E}}
\newcommand{\pa}[1]{\PP (A_{#1})}
\newcommand{\cir}{\partial \B}
\newcommand{\CC}{{\mathbb C}}

\newcommand{\clus}[2]{{\cal C} ({#1},{#2})}
\newcommand{\comment}[1]{}
\newcommand{\B}{{\mathcal B}}
\begin{document}

\title{Random soups, carpets and fractal dimensions}
\author{\c Serban Nacu${\null}^a$ \and Wendelin Werner\footnote {Research supported by the Agence Nationale pour la Recherche under the grant ANR-06-BLAN-0058}$ \hskip 6pt {\null}^b$}
\date {C.N.R.S${\null}^{a}$, Ecole Normale Sup\'erieure${\null}^{a,b}$ and Universit\'e Paris-Sud 11$\null^{b}$}
\maketitle

\begin {abstract}
We study some properties of a class of random connected 
planar fractal sets induced by a Poissonian scale-invariant and translation-invariant point process.
 Using the second-moment method, we show that their 
Hausdorff dimensions are deterministic and equal to their 
expectation dimension. We also estimate their low-intensity limiting behavior. 
This applies in particular to the ``conformal loop ensembles'' defined via Poissonian clouds of Brownian loops
 for which the expectation dimension
 has been computed by Schramm, Sheffield and Wilson.
\end {abstract}


\section{Introduction}

In this paper, we study certain random planar fractals that are close in spirit 
to random Cantor sets constructed via independent iterations.
Before describing the class of sets that we will focus on, let us first 
recall some features of  the ``classical'' planar random Cantor sets,
 sometimes known as Mandelbrot percolation or fractal percolation:

\medbreak

{\bf Mandelbrot percolation.}
Define a set $F$ by removing independently each dyadic square 
inside the unit square $[0,1]^2$ with probability $p$. 
This self-similar iterative procedure defines a random fractal that has been studied extensively \cite {Ma, MR, CCD, Fa, BC0}. 
In the case where one replaces dyadic by triadic, one gets a natural ``random Sierpinski carpet''.
Clearly, one can interpret the set $F$ as the limit of a Galton-Watson tree. This implies immediately that when $p \ge 3/4$, then $F$ is almost surely empty, and that when $p < 3/4$, then $P ( F \not= \emptyset )  > 0$.
It is then possible and easy to compute the Hausdorff dimension of $F$ as a function of $p$  (see for instance \cite {Fa}). 

In fact, when $p$ is small enough and the set $F$ is non-empty, $F$
has non-trivial connected components.
It seems clear that their dimensions must be a deterministic function of $p$, 
but in general it appears not possible to compute it explicitely in terms of $p$. 

The set $F$ is statistically invariant under dyadic scaling 
in a rather obvious sense: for a dyadic square $S= [0, 2^{-n}]^2$, 
the law of $2^n (S \cap F)$ conditioned by the event that it is non-empty 
is equal to the law of $F$ itself conditioned not to be empty. 
A similar invariance under dyadic translations can be stated. However, this 
invariance is restricted to ``dyadic transformations'' and it cannot 
be extended to more general maps.

\medbreak
{\bf Poisson models.}
A natural way to obtain stronger scale
and translation invariance is to define the set $F$ by
removing from $[0,1]^2$ all squares of a statistically 
translation-invariant and scale-invariant Poisson point process of squares. 
In fact, one could also replace these squares by other planar shapes, such as line segments or disks.
This gives rise to self-similar Poisson percolation models, 
as studied for instance in \cite {ZS1,ZS2,BC}. 
In these papers, the focus is on the existence and the nature of the phase transition for 
the connectivity property of $F$ in terms of the intensity $c$ of the Poisson point process (that replaces the factor $\log (1/(1-p))$):
For small $c$ (i.e.  small $p$), the set $F$ can have non-trivial connected components, and one can define 
 the ``carpet'' $G$ that consists of all points of $(0,1)^2$ that can be connected 
to the boundary of the unit square by a path that remains in $F$
(a typical point in $F$ will in fact not satisfy this property; see 
Figure~\ref{square_soup} for an example, and section~\ref{remaining} 
for a more complete discussion).

\begin{figure}\label{square_soup}
\centerline{\includegraphics*[height=3in]{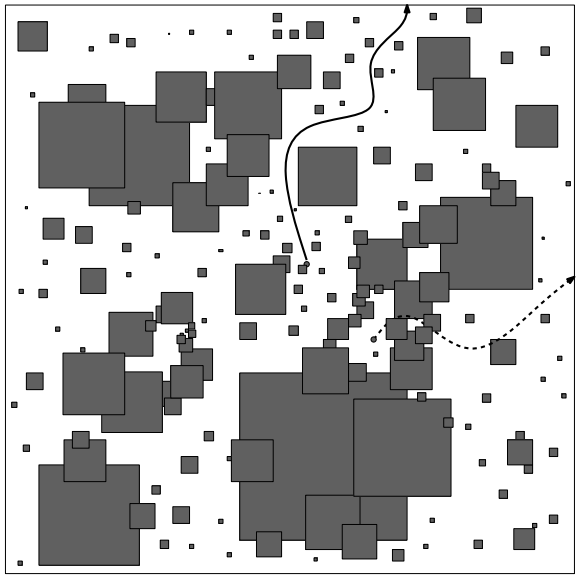}}
\caption {
Sketch of a Poisson square soup. The point at the end of the solid
line can be connected to the boundary by a path that does not cross
the soup, and thus belongs to the carpet. The point at the end of
the dotted line is completely surrounded by the soup and thus is not
in the carpet. 
}
\end {figure}

\medbreak
{\bf A special case: Brownian loop soups.}
Recently, it has been pointed out that a special and interesting case is 
to remove the interiors of Brownian loops instead of squares. 
In other words, one removes the interiors of a Poisson collection of 
Brownian loops called the Brownian loop-soup and introduced in \cite {LWls}. 
Indeed, planar Brownian motion is conformally invariant, so that this collection of loops 
is not only scale-invariant, but also conformally invariant in a rather strong sense 
(related to the conformal restriction property described in \cite {LSWrest,Wsal}). 
This allows \cite {Wcras,ShW} to derive links
with other conformally invariant objects such as the Gaussian Free Field 
or the Stochastic Loewner Evolutions (SLE), via the notion of 
Conformal Loop Ensembles (CLE) studied in \cite {Sh,ShW}. 
The link with SLE enables the study of many properties of those random Cantor sets 
(see Sheffield-Werner \cite {ShW, ShW2}):
\begin {itemize}
\item
We know the carpet is non-trivial if and only if $0 < c \le 1$ 
(the value $1$ depends of course on the choice of normalization 
for the Brownian loop measure).
\item
The complement of the carpet is made of disjoint ``holes''  whose boundaries
are SLE($\kappa$) loops for some explicit $\kappa=\kappa (c)$ 
and the dimension of these loops is known (\cite {Be,Ladim}). 
\item
Schramm, Sheffield and Wilson \cite {SchShW} have computed (as a function of $\kappa$)
the ``expectation dimension'' 
of the carpet that measures the mean number of $\epsilon$-balls needed to cover it.
\end {itemize}
This last result was our initial motivation for the present paper. 
We show here that the Hausdorff dimension of the carpet is deterministic and
equals the expectation dimension.
In combination with the connection between 
CLEs and Brownian loop soups derived in \cite {ShW,ShW2} and
the explicit formula for the expectation dimension in \cite{SchShW}, 
this completes the determination 
of the almost sure Hausdorff dimensions of CLE carpets.

More generally, our paper illustrates the fact that the ``Brownian loop-soup'' 
approach to CLE and SLE can be helpful in the derivation of second-moment estimates that are
used to determine Hausdorff dimensions (these second-moment estimates can turn out to be difficult to handle directly in the SLE setting, see 
for instance \cite {Be} for the dimension of the SLE curve itself). 
\begin{figure}\label{brownian_soup}
\centerline{\includegraphics*[height=250pt]{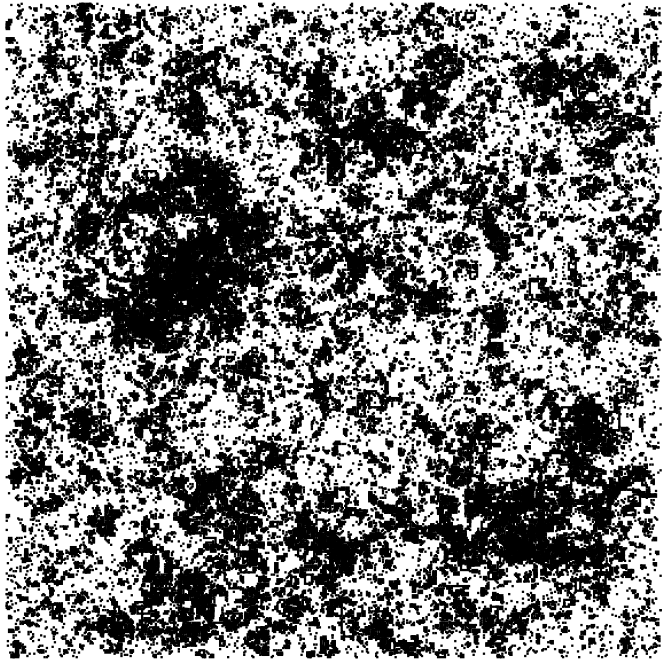}}
\caption {
Sketch of a Brownian loop soup in the unit square. Brownian loops are generated according
to a Poisson point process. Only loops above a certain diameter are
drawn; the actual loop soup is in fact dense in the square and no loop touches the boundary of the square.
This is in fact a random walk loop-soup approximation of the Brownian loop-soup, see \cite {LTF}.
}
\end {figure}
\medbreak

{\bf The contribution of the present paper.}
We study the dimension of the random carpet 
obtained from a Poisson Point process corresponding to a self-similar 
and translation invariant measure $\mu$. 
In particular, we show that the carpet's Hausdorff dimension $d(c)$ is non-random, 
that it is equal to its ``expectation dimension'', and we  interpret the first terms 
of the expansion of $d(c)$ as $c \to 0+$ in terms of $\mu$. 
The techniques that we use are rather classical, and are based on 
second-moment estimates (see e.g. \cite {MP}).

One can also note that in the Brownian loop-soup case, the first term in the expansion $2-d(c)$ when $c \to 0+$ is related to 
the expected area of the filled Brownian loop computed by \cite{GTF}.

\medbreak

The paper is organized as follows. 
In Section 2, we define  a class of invariant Poisson point processes
on planar curves and prove some of their elementary properties.
In Section 3 we show that the Hausdorff dimensions of 
random carpets defined by a random soup are deterministic,
and that first moment computations are enough to determine these dimensions.
In Section 4, we study the dimension of the carpet when $c \to 0+$ 
using an approximation of the carpet by a simpler set.

\section {Setup}

\subsection{The measures}

We are interested in measures $\mu$ on the set of compact planar curves ${\cal U}$
(a curve $\gamma$ is defined here as the image of a continuous function 
from $[0,1]$ into the complex plane). We will not need to use a strong topology on this set,
so we can just view these curves as compact subsets of the plane and 
simply use the Hausdorff topology.

 We say that the measure $\mu$ is scale-invariant (resp. translation-invariant) 
if it is invariant under the transformations $\gamma \mapsto \rho \gamma$ 
for all $\rho >0$ (resp. $\gamma \mapsto z + \gamma$ for all $z \in \CC$). 

We will focus on the case where the measure $\mu$ is {\bf ``locally finite''}
in the sense that the $\mu$-mass of the set of curves $\gamma$ with diameter greater 
than $1$ that are included in a $2 \times 2$ square is finite 
(note that we could replace $2 \times 2$ by $r \times r$ for any given $r > 1$). 
Local finiteness is in fact necessary in our setup, otherwise the carpets that we define
are almost surely empty.

We define ${\cal M}$ to be the set of all translation-invariant, 
scale-invariant and locally finite measures $\mu$ that are not ``one-dimensional''. 
More precisely, we require that for any real $\theta$, 
$$ 
\mu ( \{\gamma \ : \  \exists x , y  \in \gamma \ : \ (x-y) \notin e^{i \theta} \R \} )  \not= 0  
$$
(note that if this quantity is not $0$, then it is infinite because of translation-invariance). 
We need this condition to exclude degenerate cases such as $\mu$ being supported on line
segments parallel to the $x$-axis.

Throughout this paper, $|S|$ will denote the diameter of a bounded set $S$.
We denote by $\B(z,r)$ the ball centered at $z$ with radius $r$,
and by $\cir(z,r)$ its boundary circle.
We define ${\cal A} (S)$ to be the area of the set ``surrounded'' by $S$, 
that is, of the complement of the unbounded connected component of the complement of $S$
(if $S$ is a circle, then ${\cal A} (S)$ is the area of the corresponding disc). 
For a non-empty compact set $K$, we define $Z (K)$ as the point in $K$ 
with the smallest $x$-coordinate (if there are several of them, 
choose the one among them with smallest $y$-coordinate).

Our first result states that any measure in ${\cal M}$ can be constructed as
the product of three components that describe location, scale, and shape.

\begin {lemma} \label {l1}
\begin {itemize}
\item
Suppose that $\pi$ is a measure on the set of compact planar curves such that 
$\pi ( |\gamma|^2 ) < \infty$.
Then, consider the following product measure on 
$\R^2 \times (0, \infty) \times {\mathcal U}$:
 $$d^2 z \otimes \frac {d\rho }{\rho} \otimes \pi ( d\gamma).$$ 
If we define $\gamma' = \rho (z + \gamma)$, the previous measure induces 
a measure $\mu'$ on the set of compact planar curves defined by: 
\begin {equation}
\label {procedure}
\int F (\gamma') \mu' (d \gamma') = 
\int F ( \rho ( z + \gamma)) d^2 z \ \frac {d \rho} {\rho} \    \pi ( d\gamma) 
\end {equation}
for all measurable bounded positive $F$. 
Then, the measure $\mu'$ is translation-invariant, scale-invariant and locally finite.
\item
Conversely, any measure $\mu$ in ${\cal M}$ can be constructed in this way. 
In fact, it is always possible to construct 
$\mu$ via (\ref {procedure})  starting from a finite measure $\pi$ that is 
supported on the set of curves with diameter $|\gamma| = 1$ 
that are included in the square $[-1,1]^2$.
\item
Finally, if a measure $\mu \in {\cal M}$ can be defined starting 
from two different finite measures $\pi$ and $\pi'$ as before then 
$\pi ( |\gamma|^2) = \pi' ( |\gamma|^2)$ and 
$\pi ( {\cal A} ( \gamma ) ) = \pi' ({\cal A} ( \gamma) ) < \infty$.
We will denote this last quantity by $\beta ( \mu)= \beta (\pi)$.
\end {itemize}
\end {lemma}

Note that defining $\gamma'= \rho (z +  \gamma)$ as above, 
or $\gamma''=  z'' +\rho  \gamma$ under 
$d^2 z'' \otimes (d \rho / \rho^3) \otimes \pi (d \gamma)$ is the same 
(via the change of variable $z''=\rho z$). 
That is, we can either translate first or scale first, but the distribution
of the scaling factor differs in the two cases.

\begin {proof}
\null
\begin {itemize}
\item
Let us suppose that $\pi$ is a measure on the set of compact planar curves 
such that $\pi ( |\gamma|^2 ) < \infty$.
Then, the measure $\mu'$ defined via (\ref {procedure}) is clearly 
scale-invariant and translation invariant. We have to check that it is locally finite. 
Let us compute the $\mu'$-mass of curves that fall in the square 
$[-1,1]^2$ and have diameter greater than $1$.
Note first that
 $$\int 1_ {|\rho \gamma | \ge 1 } \frac {d \rho}{ \rho^3 }  \pi (d \gamma) 
 = \pi \left( \int_{1/|\gamma|}^{\infty} \frac {d \rho}{  \rho^3 } \right) 
 =  \pi ( |\gamma|^2 ) / 2.$$ 
But for each given $\gamma$, the Lebesgue measure of $\{z \ : \ Z( z + \gamma) \in [-1, 1]^2 \}$  
is $4$. Hence, integrating the previous identity 
over $z$, it follows that the $\mu'$ mass of the set of curves that have
diameter greater than $1$ and are subsets of $[-1,1]^2$ is bounded by 
$4 \pi ( |\gamma|^2) / 2$. This ensures that $\mu' \in {\cal M}$. 

\item
Conversely, suppose that $\tilde \mu \in {\cal M}$. 
We would like to find a corresponding measure $\pi$.
Intuitively $\pi$ should be the ``projection'' of $\tilde \mu$
on the set $K$ of curves $\gamma$ that have diameter $1$ and such that the point  
$Z (\gamma)$  is at the origin.
We can write any curve $\gamma = |\gamma| \cdot \gamma_0 + Z(\gamma)$
where $\gamma_0 \in K$, and this induces a representation of the measurable
space ${\mathcal U}$ as a product ${\R}^2 \times K \times {\R}$.
Therefore, if we define, for all translation-invariant and scale-invariant functions $F$
on planar curves,
$$ \pi ( F) = \frac {2}{3} 
\tilde \mu ( F ( \gamma ) 1_{| \gamma | \in [1,2] } 1_{ Z ( \gamma) \in [-1,1]}), $$
this induces a measure $\pi$ on $K$.
 Since $\tilde \mu$ is locally finite, the measure $\pi$ has finite mass 
(take $F=1$ in the previous expression). 
 We then define $\mu$ from $\pi$ as in (\ref {procedure}).  
For $F$ translation-invariant and scale-invariant, we have
\begin{eqnarray*}
\mu ( F ( \gamma ) 1_{| \gamma | \in [1,2] } 1_{ Z ( \gamma) \in [-1,1]})
& = &
\int F ( z+\rho\gamma ) \, 1_{| z+\rho\gamma | \in [1,2] } \, 1_{ Z ( z+\rho\gamma) \in [-1,1]}
\, d^2 z \times \frac {d\rho}{ \rho^3} \times \pi(d\gamma) \\
& = &
\int F (\gamma ) \, 1_{\rho \in [1,2] } \, 4
\times \frac {d\rho}{\rho^3} \times \pi(d\gamma) \\
& = &
\frac 3 2 \pi(F) \\
& = &
 \tilde \mu \bigl( F ( \gamma ) 1_{| \gamma | \in [1,2] } 1_{ Z ( \gamma) \in [-1,1]}\bigr) .
 \end{eqnarray*}

 Clearly, the measure $\mu$ is also scale-invariant and translation-invariant. 
It follows that for any $f$ and $g$, 
 $$ \mu ( F( \gamma) f( |\gamma|) g ( Z (\gamma) )) = 
\tilde \mu ( F( \gamma) f( |\gamma|) g ( Z (\gamma) ))$$
and finally, using the product representation of $\mathcal U$, we get $\mu=\tilde \mu$.
 
\item
To check the final statement, compute
the $\mu$-mass of the set of curves $\gamma$ 
such that $Z ( \gamma) \in[0,1]^2$ and ${\cal A} ( \gamma) \ge 1$ 
in terms of $\pi$. This equals
$$ 
\int_{[0,1]^2} d^2 z \times \int_0^\infty \int_{\cal U} 
\frac {d \rho}{\rho^3} d\pi ( \gamma) 1_{{\cal A} (\rho \gamma) \ge 1 } 
= \pi \left( \int_{1 / \sqrt {{\cal A} (\gamma)} }^\infty 
\frac {d \rho}{\rho^3} \right) 
= \pi ( {\cal A} ( \gamma) )/ 2.
$$
This last quantity is therefore the same for $\pi$ and $\pi'$. 
Exactly the same argument implies that $\pi ( |\gamma|^2) = \pi' ( |\gamma|^2)$
and in fact works for any function of $\gamma$ that scales with $|\gamma|^2$.
Since any locally finite $\mu$ can be constructed from a measure $\pi$
supported on curves of diameter 1, it follows that $\pi ( |\gamma|^2 )$ is finite.
\end {itemize}
\end {proof}

\medbreak
It is worth emphasizing that the scale-invariance of the measures 
$\mu \in {\cal M}$ is with respect to the transformations 
$\gamma \mapsto \lambda \gamma$ that
move the point $Z(\gamma)$ away from the origin when $\lambda$ is large. 
One could also study how $\mu$ behaves under the transformations 
$$ T_\lambda : \gamma \mapsto Z(\gamma) + \lambda ( \gamma - Z (\gamma))$$
that magnify $\gamma$ without changing its lowest-left-most point. 
The lemma in fact implies that the image of $\mu$ 
under $T_\lambda$ is $\lambda^2 \mu$.

\subsection {Soups}

Recall that a Poisson point process with intensity $\mu$
where $\mu \in \cal M$ is a random countable collection of curves
 $\Gamma = (\gamma_j, j \in J)$ 
in the plane such that for any 
disjoint (measurable) sets of curves $A_1, \dots, A_n$, 
the numbers $N(A_1), \ldots, N (A_n)$ of 
curves in $\Gamma$ that are respectively in $A_1, \ldots, A_n$ 
are independent Poisson random variables 
with respective mean $\mu (A_1), \ldots, \mu (A_n)$ (when $\mu (A_j)= \infty$, then $N(A_j)= \infty$ almost surely).
We call this a {\bf random soup} in the plane with intensity $\mu$.

For any domain $D \subset \CC$, we define the random soup in $D$ as the set
of all curves in the soup $\Gamma$ that are contained in $D$. 
In other words, if $J_D= \{ j \in J \ : \ \gamma_j \subset D \}$, 
then the soup in $D$ is $\Gamma_D= ( \gamma_j, j \in J_D)$. 
Note that $\Gamma_D$ is itself a Poisson point process with intensity 
$\mu_D = \mu 1_{\gamma \subset D}$.

The properties of $\mu$ ensure that $\Gamma$ is 
translation-invariant and scale-invariant. The fact that $\mu$ is locally finite 
implies that almost surely  for each bounded domain $D$ and each $r>0$, $\Gamma_D$ 
contains only a finite number of curves of diameter greater than $r$. 

A few examples of natural measures $\pi$ and their corresponding random soups are the following:
\begin{itemize}
\item The measure  $\pi$ is a constant times the law of a Brownian loop of time-length 1: this defines
the {\bf Brownian loop soups} introduced in \cite {LWls}. 
It is immediate to check that the measure $\mu$ is locally finite 
in our sense (i.e. just check that $\pi ( |\gamma|^2 ) < \infty)$. 
\item
We can also look at the law of the outer boundary of the Brownian loop of time-length $1$. 
This is in fact a SLE($8/3$) loop that also corresponds to scaling limits of 
critical percolation cluster outer boundaries \cite {Wsal}.
It defines a soup of ``outer boundaries of Brownian loops''. 
Local finiteness is a consequence of that of the Brownian loop measure.

\item The measure $\pi$ is a constant times the Dirac mass supported on the unit circle
(so $\mu$ is also supported on circles): this is the {\bf circle soup}.

\item The measure $\pi$ is a constant times the law of the segment with endpoints $u$ and $-u$,
where $u$ is chosen uniformly at random on the unit circle (so $\mu$ is supported
on segments): this is the {\bf stick soup}. A variant of this is the
{\bf discrete stick soup}, where $u$ is a uniformly chosen random vertex of
a regular $n$-gon.
\end{itemize}

We now define a condition on $\mu$ that will roughly prevents the soups from being too ``dense''.

\begin{lemma}\label{mu_cross}
Let $L_R$ be the set of curves that intersect the unit disk
and have diameter at least $R$. Suppose that $\mu \in {\cal M}$. 
Then the following statements are equivalent:
\begin {itemize}
\item
For some $R> 0$, $\mu(L_R) < \infty$.
\item
For all $R> 0$, $\mu (L_R) < \infty$.
\item 
The measure $\mu$ is constructed as in Lemma \ref {l1} via (\ref {procedure}) from a finite measure $\pi$ 
such that $\pi ( |\gamma|^2) < \infty$ holds and that satisfies
$$ \int_0^1 \frac {dr}{r}  \times \pi ( \int d^2 z \ 1_{d ( z, \gamma) \le r })  < \infty.$$ 
\end {itemize}
When the statements hold, we say that the corresponding soup is {\bf thin}.
\end {lemma}

Note in particular 
that if a soup is thin, then the (two-dimensional) Lebesgue measure of $\gamma$ is $\pi$-almost surely (and therefore $\mu$-almost surely)
equal to zero (otherwise the third statement would not hold) 
so that a given point in the plane belongs almost surely to no curve of the corresponding soup. 

\begin{proof}
Suppose that $\mu \in {\cal M}$ is defined via (\ref {procedure}) 
from a measure $\pi$ supported on the set of curves $\gamma$ 
with $\pi ( |\gamma|^2) < \infty$ and $Z ( \gamma)= 0$.
The curve
$z + \rho \gamma$ has diameter $\rho |\gamma|$, and it intersects
the unit disk if and only if  $-z / \rho$ is at a distance less than $1/\rho$
from $\gamma$. Hence, for all $R$, we have
\begin{eqnarray*}
\mu(L_R) 
& = & \int d^2 z \int_R^\infty \frac {d \rho}{\rho^3} \int \pi ( d \gamma) 1_{ d(0, z+ \rho \gamma) \le 1 } \\
& = & \int d^2 z \int_R^\infty \frac {d \rho}{\rho^3} \int \pi ( d \gamma) 1_{ d(z, \rho \gamma) \le 1 } \\
& = & \int_R^\infty  \frac {d \rho}{\rho^3} \int \pi ( d \gamma)  \int\rho^2 \   d^2 z' \  1_{ d(z' , \gamma) < 1/ \rho } \\
& = & \int_0^{1/R} \frac {dr}{r} \pi ( \int d^2 z \ 1_{ d (z,\gamma) \le r } ) 
\end{eqnarray*}
The equivalence between the three statements follows readily 
(note that this integral in $r$ can diverge only near $r=0$).
\end {proof}

\begin{corollary}\label{thin_large_R}
\begin{itemize}
\item[(i)] If a soup is thin, then $\lim_{R \to \infty} \mu(L_R) = 0$.
\item[(ii)]
A random soup is thin if and only if almost surely, for every ring in the plane, 
only a finite number of curves in $\Gamma$ do intersect 
both the inner circle and the outer circle.
\end{itemize}
\end{corollary}

\begin {proof}
The first statement follows from the dominated convergence theorem.

For the second part, let $S_{a,b}$ be the set
of curves that intersect both circles of radius $a$ and $b$.
It follows easily that for $R>1$, $L_{R+1} \subset S_{1,R} \subset L_{R-1}$.
If $S_{1,R}$ is finite, then so is $L_{R+1}$, and Lemma~\ref{mu_cross}
implies that the soup is thin. Conversely, if the soup is thin,
then $L_{R-1}$ is finite, so all $S_{1,R}$ are finite, and by scale
invariance so are all $S_{a,b}$.
The rings with rational radii form a countable dense set among all rings,
and this completes the proof.
\end {proof}

The circle soups, the square soups and the stick soups are obviously thin. 
The third condition in the lemma shows that as soon as the mean area of the 
$r$-neighborhood of $\gamma$ (defined under $\pi / |\pi|$) 
decays for instance faster than $1/\log (1/r)^2$ as $r \to 0$,  the 
corresponding $\mu$-soup is thin.
For instance, if $\pi$ is supported on curves with Hausdorff dimension $d<2$,
the size of the neighborhood decays like $O(r^{2-d})$ and the soup should be 
thin, assuming the bound holds in expectation.

The Brownian loop-soup is not thin, but we now show that the soup 
of its outer boundaries is thin 
(and this will be enough for our purposes since they will define the same carpets).

\begin {lemma}
The soup of outer boundaries of Brownian loops is thin.
\end {lemma}

\begin {proof}
Since outer boundaries $\gamma$ of Brownian loops have dimension $4/3$ 
(see \cite {LSW2}),
the area of their $r$-neighborhoods decays typically at least like 
$r^{2/3+ o (1)}$ when $r \to 0$.
However, Lemma~\ref{mu_cross} requires looking at expectations, rather than
at ``typical'' behavior, so additional arguments are needed.
Here is a short self-contained proof that does not rely on \cite {LSW2}.

Consider the Brownian loop-measure, defined on
Brownian loops $(Z_t, t \in [0,1])$ 
of time-length $1$ that start and end at the origin. 
A point $z$ is in the $r$-neighborhood of the 
outer boundary of the loop iff it lies at distance at most $r$ of the loop, 
and the disc $\B (z,r)$ is not disconnected 
from $\infty$ by the loop.
 
Using circular re-rooting of the loop, the mean area of the $r$-neighborhood of the 
outer boundary is clearly bounded 
by four times the mean area of the set of points $z$ that lie at distance 
at most $r$ from $Z [0,1/4]$ and such that $\B(z,r)$ is not disconnected 
from infinity by $Z[0,1/2]$. 

It is easy to check that the law of $Z [0,1/2]$ is absolutely continuous with respect 
to that of $B [0,1/2]$, where $B$ is a standard Brownian motion, 
and that its Radon-Nikodym derivative is bounded. 
Hence, it is sufficient to bound the expected area of the set of points 
$z$ that lie at distance less than $r$ of $B [0,1/4]$ 
and such that $\B (z,r)$ is not disconnected from infinity by $B[0,1/2]$. 
Using the strong Markov property of $B$, it follows immediately that 
this quantity is bounded by the mean area $m$ of the $r$-neighborhood of $B[0,1/4]$ 
times the probability that a planar Brownian motion started at distance $r$ from the origin 
does not disconnect $\B(0,r)$ before time $1/4$. This last probability is bounded by $r$ 
to some positive power when $r \to 0+$ (see any introductory paper on disconnection exponents --
 this is due to the fact that one the one hand the probability that it stays in $\B (0, \sqrt {r})$
 during the time-interval $[0, 1/4]$ is very small, and on the other hand that the probability that 
it does not disconnect $\B (0,r)$ before reaching the circle of radius $\sqrt {r}$ decays at least 
like a positive power of $r$ when $r \to 0+$), 
and the former mean area $m$ is bounded. Hence, the thinness of the soup of 
outer boundaries of Brownian loops follows.   
\end {proof}

\subsection {Carpets and loop clusters} \label{carpet_cluster}

We will be interested in the fractal carpets defined using random soups. 
To simplify the discussion, we assume in this section that 
the measure $\pi$ (and therefore also $\mu$) is supported on {\bf simple} loops, 
that is, injective continuous maps from the unit circle into the plane. 
This will exclude the discrete stick soups 
(discussed briefly at the end of this section) 
and the Brownian loop soup (but we can define and recover the carpet for the latter
via soups of outer boundaries of Brownian loops).

The carpet corresponds to a connected component of the random Cantor set 
obtained by removing the interiors of all the loops. 
There are {\it a priori} various ways to define it. 

Consider a random soup $\Gamma_D$ in a simply connected domain $D$ (such that $D \not= \CC$),
consisting of loops $(\gamma_j, j \in J)$.
Each $\gamma_j$ defines an interior $O_j$ 
(the bounded connected component of $\R^2 \setminus \gamma_j$). 
The set $F={\overline D} \setminus \cup_j O_j$ is then a random closed subset of $D$ that we shall
sometimes refer to as the ``remaining set''. Our interest is in its connected components:

\begin{Definition}
The {\bf carpet} $G$ is the set of points $z$ in $D$ such that there exists 
a continuous path from any neighborhood of $z$ to $\partial D$ that stays in $F$. 
\end{Definition}

Loosely speaking, the carpet is the connected (by arcs) component of 
$F$ that has $\partial D$ as part of its boundary. Note that, for technical
reasons, we allow the connecting paths to intersect the loops in the soup
(but not their interiors). We will comment in a moment on whether this definition is equivalent to 
saying that there exists a continuous path from $z$ itself to $\partial D$ that stays in $F$.

\medbreak
Another almost equivalent approach is to look at {\bf clusters of loops}.
We say that two loops $\gamma$ and $\gamma'$ of a random soup are connected 
if one can find a finite sequence 
$\gamma_0= \gamma, \gamma_1, \ldots , \gamma_n = \gamma_n'$ 
in the soup such that for all $j \le n-1$, 
$$ \gamma_j \cap \gamma_{j-1} \not= \emptyset \hbox { and } O_j \cap O_{j-1} \not= \emptyset.$$
Clearly, connection between loops forms an equivalence relation, and
one can then define the clusters of the soup as the union of all $\gamma_j$'s 
for the loops in the same equivalence class. Note that for technical reasons,
we require here not only that 
the two curves $\gamma_j$ and $\gamma_{j-1}$ intersect but also that their interiors do. 
In many cases, this definition can be relaxed, as translation invariance 
can be easily be used to show that almost surely (for a sample of the soup), 
if two curves intersect then so do their interiors.
The clusters may be nested, this may occur for example when one loop lies
inside another one.

Given a loop-soup in the entire plane, one can wonder whether there exist 
clusters of infinite diameter. Kolmogorov's $0-1$ law implies in the standard way 
that this event has probability either 0 or 1, depending on the law of the soup.
Furthermore, if the measure $\mu$ is supported on loops with positive inner area,
it follows immediately that any disk $\B(0,R)$ is almost surely
contained in the interior of some loop $\gamma_j$.
Since this holds for arbitrarily large $R$, it follows that if there exists 
a cluster of infinite diameter, then it is unique. 
In this case, let $X$ denote the distance between the unbounded cluster and the origin. 
Clearly, $X$ is a scale-invariant finite real random variable, so that $X=0$ almost surely. 
It follows that for a loop soup in the entire plane:
\begin {itemize}
\item Either all clusters are bounded almost surely
\item Or there almost surely exists exactly one cluster, and this cluster is dense and unbounded.
\end {itemize}

Since our main interest in the present paper is the geometry of the clusters rather 
then the phase transition,
we will focus on the case when all clusters are bounded almost surely 
and the carpets are not empty. 
More precisely, our assumption on $\mu$ goes as follows:
 
\medbreak
\noindent
{\bf Subcriticality assumption 1.}
{\sl All clusters in the full-plane loop soup are almost surely bounded.}

\medbreak
In fact, we will assume an {\it a priori} slightly stronger condition:
\medbreak
\noindent
{\bf Subcriticality assumption 2.}
{\sl 
With positive probability, there exists a (random) closed loop $\ell$ in the plane 
that surrounds the origin and does not ``cross'' any loop of the loop-soup.}

\medbreak
In this definition, and throughout the rest of this paper, we say that two loops ``cross'' if each one of the two does intersect the interior of the other one.
Let us stress that the loop $\ell$ is not a loop of the loop-soup.
\medbreak

Note that a simple $0-1$ argument then implies that any given point is almost surely surrounded by infinitely many such loops (of arbitrarily small or large diameter) that cross no loop in the loop-soup. Let us now consider a domain $D$ as before and choose a given point $z$ on its boundary. The fact that there almost surely exist such small loops $\ell$ around $z$ that cross no loop of the full-plane loop-soup implies immediately that the carpet in $D$ is not empty (because the intersection of this small loop with $D$ is in the carpet). Hence, subcriticality implies that the carpets are almost surely non-empty. Note also that subcriticality clearly implies that the measure $\mu$ is thin. 

It could be interesting to study the converse i.e. whether non-triviality of the carpet 
and thinness imply our subcriticality assumption.
This is for instance not difficult if one assumes invariance of $\mu$ 
under certain rotations.
In that case, subcriticality can be proven using FKG-type arguments,
but this is not the purpose of the present paper. 
In section~\ref{sec_intensity} we review a coupling argument that shows
that most soups of interest are subcritical as long as their intensity
parameter lies below a certain value.

Furthermore, note that the arguments presented in \cite {ShW2} in order to prove that in the 
``subcritical phase,'' outer boundaries of clusters of Brownian loops are 
indeed loops (this is essentially only based on an FKG argument) can be easily generalized to the present setting. This implies that the subcriticality assumption 1 implies the subcriticality assumption 2 in most cases of interest (because the outer boundary $\ell$ of a loop-soup cluster that surrounds the origin is one closed loop that satisfies the conditions of the second assumption). It also indicates that in most cases, the carpet is indeed the set of points that are connected to $\partial D$ by a continuous path in $F$.

Finally, let us mention that our definition of the carpet can in fact be easily adapted to the case of the stick soups. One just needs to define the set $G$ as the set of points $z$
such that there exists a continuous path from any neighborhood of $z$ to $\partial D$ that does not ``cross'' any stick (in an appropriate sense) and to modify the definition of loop-clusters similarly. The results of the present paper would still apply.  
Other possible variants could include the possibilities that $\gamma_j$'s are discontinuous etc.

\section {Dimensions}

\subsection {Preliminaries}

In this section we will consider a thin translation- and scale- invariant random soup
on a bounded simply connected domain $D$ and we will suppose furthermore 
that the subcriticality assumption holds.
We will show that the Hausdorff dimension $\dim (G)$ of the carpet is deterministic, and that it is 
described by ``first moment estimates''.

We will use the standard second-moment method to evaluate the dimension
of a random fractal set C that is closely related to the carpet $G$. The idea
is to define a sequence of sets $C_\eps$ that converge to $C$,
to obtain probability estimates for these sets and to show that these
yield the dimension of $C$. We will use the following standard fact about first and second moments (see \cite{BeThese}
for this precise statement, or \cite {MP} for almost equivalent ones):

\begin{lemma}
\label {bef}
Let $D$ be a bounded domain and $( C_\eps, {\eps > 0})$ a  family of random Borel
subsets of $D$, so that $ C_\eps \subset C_{\eps'} $ if $\eps<\eps'$.
Define $C= \cap_{\eps >0} C_\eps$. 
Suppose that $\alpha >0$ and define the following three conditions: 
\begin {enumerate}
\item There exist positive constants
$k_1$ and $k_2$ such that 
for any small $\eps>0$ and $x \in D$, 
$k_1 \eps^\alpha \le \PP(x \in C_\eps) \le k_2 \eps^\alpha$.
\item
There exists a positive constant
$k_3$ such that 
for any $\eps>0$ and $x, y \in D$,
$\PP(x, y \in C_\eps) \le k_3  \eps^{2\alpha} |x-y|^{-\alpha}$.
\item
There exists  a positive constant
$k_4$ such that 
for all $\eps>0$ and $x \in D$, the expected area of $C_{\eps} \cap \B(x,\eps)$,
conditional on the event that $x \in C_\eps$,
is at least  $k_4 \eps^2$.
\end {enumerate}
Then:
\begin {itemize}
\item If $\alpha \le 2$ and both 1. and 3. hold, then $\dim (C) \le 2 - \alpha$ almost surely.
\item If $\alpha \le 2$ and both 1. and 2. hold, then $\dim (C) \ge 2 - \alpha$ with positive probability.
\item If $\alpha > 2$ and both 1. and 3. hold, then $C$ is almost surely empty.
\end {itemize} 
\end{lemma}

In other words, if we have  first and second moments estimates 
for the area of $C_\eps$ with the correct asymptotics, then we can
control the dimension of $C$. 

We will obtain such estimates for a random set $C$ related to the carpet, and then use a 0-1 law
to argue that the dimension is deterministic, and therefore
is almost surely equal to $2- \alpha$.
One natural choice for $C_\eps$
would be the set of all points within $\eps$ of the carpet;  however,
a slightly different definition is better suited in our setting:

\begin{Definition}
Let $C_\eps$ be the set of points $x \in D$ with
the property that there exists a path connecting $x$ to
$\partial D$ that does not cross any interior of a curve of the soup 
$\Gamma_{D \setminus \B(x,\eps)}$.
In other words, we take all curves in the soup, we ignore the ones that
get within distance $\eps$ of $x$, and we look if there exists a path 
connecting $x$ to the
boundary that does not hit the interior of any of the remaining curves.
We define the {\bf approximate carpet} $C = \cap_\eps C_\eps$. 
\end{Definition}

\begin {lemma}
Let $K = \cup_{\gamma \in \Gamma_D} \gamma$ be the union of all
curves in $\Gamma_D$. Then we have
$$ C \setminus K \subset G \subset C $$
\end {lemma}

\begin {proof}
 Clearly, if a point $z$ is in the carpet $G$, then it is in $C$. 
Conversely, if a point is in $C$ and is not on any curve $\gamma_j$, 
then it necessarily is in the carpet.
Indeed, take $z$ in $D$ but not on any curve, and let $\eps>0$.
Because the soup is thin, only finitely many curves 
$\gamma_1, \ldots, \gamma_n$ in the soup intersect both
circles of radii $\eps/2$ and $\eps$ centered at $z$
(Corollary~\ref{thin_large_R}).
Hence for 
$$\delta = (\eps/2) \wedge \min_{1\le i \le n} d(z, \gamma_i),$$
no curve in the soup that intersects the circle $\cir(z,\eps)$
comes $\delta$-close to $z$.
If $z$ is in $C$, then by looking at $C_\delta$, 
we conclude there exists a path joining $\cir(z,\eps)$
to $\partial D$ that crosses no curve
in the soup. Since this is valid for all $\eps$,
we conclude that $z$ is in the carpet.
\end {proof}

\subsection {First moment estimates for the disk}

We first restrict ourselves to the case  where $D$ is the unit disk
$\U$. 
Let $\Gamma$ be the random soup in the unit disk.
For any $\eps \in (0, 1)$, let $\geps$ be the set of curves in $\Gamma$ that
are contained inside the ring $\{ \eps < |z| < 1 \}$
(including curves that wind around the inner circle
without touching it).
Let $A_\eps$ be the event that $\{0 \in C_\eps \}$ i.e. there
exists a path that connects the origin to the unit circle and does not
cross any curve in $\geps$.
We first  show  that $\pa{\eps}$ behaves asymptotically
like some power of $\eps$. 

\begin{lemma}\label{pa_upper} 
There exist $k > 0$  and $R>2$ such that for any $\eps, \eps' \in (0, 1/R)$,
\begin{equation}
k \pa{\eps} \pa{\eps'/R} \le \pa{\eps \eps'} \le \pa{\eps} \pa{\eps'}
\end{equation}
\end{lemma}

\begin{proof}
The upper bound is trivial: Let $\Gamma'$ be the set of curves in $\Gamma$ 
that are contained
inside the ring $\{ \eps\eps ' < |z| < \eps \}$.
Let $E$ be the event that there is a path that connects the circles
$\cir(0,\eps\eps')$ and $\cir(0,\eps)$ and does not cross any curve in
$\Gamma'$. By scale invariance of the soup, $\PP(E) = \pa{\eps'}$.
Clearly $A_{\eps\eps'} \subset E \cap A_\eps$.
Since $\Gamma'$ and $\geps$ are
disjoint, $E$ and $A_\eps$ are independent.
Hence
$\pa{\eps \eps'} \le \PP(E) \PP(A_\eps) = \pa{\eps} \pa{\eps'}$.

Note that either 
$\pa {\eps} > 0$ for all $\eps<1$, or $\pa {\eps}= 0$ for all $\eps$
smaller than some $\eps_0$.
It is easy to verify that in this second case, the carpet is almost surely empty. 
We can therefore assume we are in the first case.
Note also that if almost all loops for $\mu$ have nonempty interiors,
then with positive probability, one loop of the loop-soup will surround the origin, so
$\pa {\eps} < 1$ for $\eps$ small enough.

Let us define the event $B_R( r)$ that in the ring 
$\{z \ : \ r < |z| < Rr \}$, there exists a closed loop that 
surrounds the origin and that crosses no curve in the soup 
(we consider here the entire soup in the plane). 
Note that because of scale-invariance, the probability $b(R)= P ( B_R(r))$ 
does not depend on $r$. Furthermore, because soup clusters are bounded,
the subcriticality assumption implies that for sufficiently small $r$ and large $R$, 
$P (B_R (r)) >0$ (consider, for example, the boundary of the
cluster containing the origin).
Hence, $b(R)$ is positive for large enough $R$. 

Let us now fix $R > 2$ such that $b (R/2)>0$, and  
choose $\eps$ and $\eps'$ in $(0, 1/(2R))$.
Consider the following four events:
\begin {itemize}
\item
$E_1$ is the event that for the soup in the ring $\{z \ : \ \eps < |z | < 1 \}$, 
there exists a path joining the inner boundary to the outer boundary of the ring, 
that does not cross any loop of this soup.
The probability of this event is $\pa { \eps}$.
\item
$E_2$ is the event that for the soup in the ring 
$\{ z \ : \  \eps \eps' < |z | < R \eps \}$, 
there exists a path joining the inner and outer boundary of the ring, 
that does not cross any loop of this soup. 
By scale-invariance, the probability of this event is $\pa {\eps' /R}$.
\item
$E_3$ is the event that in the ring 
$\{ z \ : \ (4/3) \eps    < | z | < (3/4) R \eps \}$, 
there exists a closed loop surrounding the origin that 
does not cross any curve in the entire soup $\Gamma_\C$. 
Because of scale-invariance, the probability of this event 
does not depend on $\eps$ and is equal to $b = b(9R/16)$.
Note that because 
$$\eps \eps' < \eps < (4/3) \eps < (3/4) R \eps < R \eps < 1,$$
this closed loop must intersect the paths described 
in the definitions of $E_1$ and $E_2$.
\item
$E_4$ is the event that no curve in the soup $\Gamma_\C$ that 
intersects the ring    $\{ z \ : \ \eps < | z | < R\eps \}$ 
has diameter greater than $\eps /4$.
Because of scale-invariance, its probability $b'$ does not depend on $\eps$. 
It is positive because the soup is thin.

Recall that an event $A$ depending on the realization of a soup 
is said to be decreasing if 
$\Gamma \notin A$ and $ \Gamma \subset \Gamma'$ implies $\Gamma' \not\subset A$. 
It is standard that decreasing events are positively correlated
(this is the FKG-Harris inequality, see e.g. \cite {J}).
Here, the events $E_1$, $E_2$, $E_3$ and $E_4$ are all decreasing. 
Therefore,
$$ P ( E_1 \cap E_2 \cap E_3 \cap E_4 ) 
\ge \prod_{j=1}^4 P ( E_j) = b b' \pa {\eps} \pa {\eps' /R}.$$
On the other hand, $ E_1 \cap E_2 \cap E_3 \cap E_4 \subset A_{\eps \eps'} $.
Indeed, if all events occur,
then we can concatenate a part $\eta_2$ of the crossing defined by $E_2$ 
to a part $\eta_3$ of the loop defined by $E_3$ 
to a part $\eta_1$ of the crossing defined by $E_1$ 
to construct a crossing $\eta$ of the ring $\{z \ : \  \eps \eps' < |z| < 1 \}$.
Then $E_4$ implies that $\eta$ does not cross any curve in $\Gamma_{\eps\eps'}$.
\end {itemize}
The lower bound follows.
\end {proof}

Such a lemma implies classically up-to-constants estimates:

\begin{corollary}\label{pa_alpha_zero}
For some positive constants $\alpha$ and $k'$, we have
\begin{equation}
\eps^\alpha \le \PP (A_\eps) \le  k' \eps^\alpha. 
\end{equation}
\end{corollary}

\begin {proof}
Let $f(\eps) = \PP(A_\eps)$. Since $f( \eps \eps') \le f ( \eps) f( \eps')$, 
it follows readily from standard subadditivity that 
$$\lim_{\eps \to 0} (\log f( \eps) / \log \eps)= \inf_{\eps< 1/R} (\log f( \eps) / \log \eps) < \infty.$$
On the other hand, if we define $g (\eps) = k f ( \eps/R)$, we get that
$$ g ( \eps \eps') = k \pa{\eps \eps'/R} \ge k^2  \pa{\eps/R} \pa{\eps'/R} \ge g( \eps) g ( \eps')$$
so that 
$$\lim_{\eps \to 0} (\log g( \eps) / \log \eps)= \sup_{\eps< 1/R} (\log g( \eps) / \log \eps) >0.$$
But 
$$\lim_{\eps \to 0} (\log g( \eps) / \log \eps)=\lim_{\eps \to 0} (\log f( \eps) / \log \eps).$$
If we define this limit to be $\alpha$, it follows that $0 < \alpha < \infty$ and that for all $\eps$, 
$$ 
\eps^\alpha \le f( \eps) \le g(R \eps) / k \le (R \eps)^{\alpha} / k .$$
\end {proof}

\subsection {Second moment estimates for the disk}

We still assume that $D$ is the unit disk $\U$. To avoid boundary
effects, it is natural to define, for any $\delta \in (0,1)$, 
the set $D_\delta$ of points
in $D$ that are at distance greater than $\delta$ from $\partial D$ i.e. 
$D_ \delta= (1-\delta) \U$.

\begin{lemma}\label{pa_alpha}
For all small $\delta$: 
\begin {itemize}
\item
There exist positive $k_1=k_1 (\delta)$ and $k_2=k_2 (\delta)$ 
such that for all $x \in D_\delta$ and for all $\eps < \delta/4$,
\begin{equation}\label{moment1}
k_1 \eps^\alpha \le \PP(x \in C_\eps) \le k_2  \eps^\alpha.
\end{equation}
\item 
There exists $k_3=k_3 (\delta)>0$ such that 
for any two points $x, y \in D_\delta$,  and all $\eps < \delta/4$,
\begin{equation}\label{pa_alpha_2}
\PP(x \in C_\eps \hbox { and }  y \in C_\eps)
\le k_3 \frac { \eps^{2\alpha} }{ |x-y|^\alpha}.
\end{equation}
\item
There exists $k_4= k_4 (\delta)$ such that for all $x \in D_\delta$ 
and $\eps < \delta / 4$,
the expected area of $C_\eps \cap \B(x,\eps)$,
given that $x \in C_\eps$, is at least $k_4 \eps^2$.
\end {itemize}
\end{lemma}

\begin{proof}

We first estimate $\PP ( x \in C_\eps)$.
Assume that $\eps < \delta / 2$. 
For any $r$ define $\Gamma_\eps^r (x)$ to be the set of all curves in $\Gamma$
contained in the ring $\B(x,r) \setminus \B(x, \eps)$. We have

$$ \B(x, \eps) \subset \B(x,\delta) \subset D \subset \B(x,2). $$

Hence if $x \in C_\eps$, then there exists a path from $\cir(x,\eps)$ to 
$\cir(x,\delta)$
that does not cross any curve in $\Gamma_\eps^\delta(x)$. Conversely,
if there exists a path from $\cir(x,\eps)$ 
to $\cir (x,2)$ that does not cross any
curve in $\Gamma_\eps^2 (x)$, then $x \in C_\eps$.
Hence both the upper and lower bounds follow from 
 Corollary~\ref{pa_alpha_zero} and scale invariance.

We now estimate the second moment. 
We can assume without loss of generality that $\eps < |x-y|/2$
(otherwise the result follows from the first moment estimate).
Consider first the case when $|x-y| \ge \delta / 2$. Define the events
\begin{itemize}
\item $E_1$: there is a path from $x$ to $\cir(x,\delta/4)$ that does not
cross any curve in $\Gamma$ contained in $\B(x,\delta/4) \setminus \B(x,\eps)$
\item $E_2$: there is a path from $y$ to $\cir(y,\delta/4)$ that does not
cross any curve in $\Gamma$ contained in $\B(y,\delta/4) \setminus \B(y,\eps)$
\end{itemize}
The balls $\B(x,\delta/4)$ and $\B(y,\delta/4)$ do not intersect, 
so $E_1$ and $E_2$ are independent. Clearly 
$\{x, y \in C_\eps\} \subset E_1 \cap E_2$,
so by scale invariance
\begin{equation}
\PP(x, y \in C_\eps) \le \PP(E_1)\PP(E_2) \le 
(4\eps/\delta)^{2\alpha} (k')^2
\end{equation}
and since $|x-y|$ is bounded below, (\ref{pa_alpha_2}) follows.

When $\beta = |x-y| \le \delta / 2$, let $z = (x+y) / 2$
and consider the events 
\begin{itemize}
\item $E_1$: there is a path $\pi_1$ from $x$ to $\cir(x,\beta/2)$ that does not
cross any curve in $\Gamma$ contained in $\B(x,\beta/2) \setminus \B(x,\eps)$
\item $E_2$: there is a path $\pi_2$ from $y$ to $\cir(y,\beta/2)$ that does not
cross any curve in $\Gamma$ contained in $\B(y,\beta/2) \setminus \B(y,\eps)$
\item $E_3$: there is a path $\pi_3$ from $\cir(z,2 \beta )$ to $\cir(z,\delta)$ 
that does not cross any curve in $\Gamma$ contained in $\B(z,\delta) \setminus \B(z,2 \beta)$
\end{itemize}
The three events involve curves contained in disjoint sets,
so they are independent. Clearly $\{x, y \in C_\eps \} 
\subset E_1 \cap E_2 \cap E_3$, and by Corollary~\ref{pa_alpha_zero},
$$ \PP(E_1) \PP(E_2) \PP(E_3) \le (k')^3 (2\eps / \beta)^\alpha 
(2\eps / \beta)^\alpha (2\beta / \delta)^\alpha $$
and (\ref{pa_alpha_2}) follows.

It now remains to check the final statement.
We show first that if $x \in C_{\eps/2}$, then $\B(x, \eps/2) \subset C_\eps$. 
Indeed, if $|y-x| < \eps/2$, then
the soup obtained by removing all curves that are at distance less 
than $\eps$ from $y$ is contained in the soup obtained by removing 
all curves that are at distance less than $\eps/2$ of $x$. 
Hence $y \in C_\eps$. 
Note also that since $D$ and $D_\delta$ are disks, we have that
for any $x \in D_\delta$ and $\eps < \delta / 4$, 
the area of the intersection 
$\B(x, \eps/2) \cap D_\delta$ is at least  $\eps^2/8$. 
Hence, 
\begin {eqnarray*}
\EE ( 1_{ x \in C_\eps } {\cal A} ( C_ \eps \cap D_\delta \cap \B(x, \eps ) ) ) 
& \ge & 
\EE ( 1_{x \in C_{\eps/2 }} {\cal A} ( C_{\eps} \cap D_\delta \cap \B(x, \eps/2))) \\
& \ge &
\EE ( 1_{x \in C_{\eps /2}} \eps^2 / 8 ) \\
& \ge & 
\PP (x \in C_{\eps /2}) \eps^2 / 8 \\
& \ge &
k_1 \eps^{\alpha + 2} / 2^{\alpha +3} 
\\
& \ge & 
(k_1 / k_2 2^{\alpha+3} ) \eps^2 \PP ( x \in C_ \eps ) 
. \end {eqnarray*}
This is precisely the last statement of the lemma.
\end{proof}

We can now apply Lemma~\ref{bef}: it shows that for any small $\delta$, 
the dimension of $C \cap D_\delta$ is almost surely not larger than $2 - \alpha$,
and is equal to $2-\alpha$ with nonzero probability.
We conclude that almost surely,
$$ \dim (C) = \dim \left( \cup_\delta (C \cap D_\delta) \right) = 
\sup_\delta  \left( \dim (C \cap D_\delta)  \right) \le 2-\alpha $$
and that 
$P ( \dim (C)  = 2 - \alpha ) > 0$.

\subsection {A 0-1 law for general domains $D$}

We now assume that $D$ is a bounded non-empty open domain.

The proof of the upper bound readily follows for the previous case.
Indeed, if $z$ belongs to the carpet defined by the soup in $D$,
and $B$ is any small disk with $z \in B \subset D$, then $z$ also
belongs to the approximate carpet defined by the soup in $B$.
Hence if we write $D = \cup B_i$ as a countable union of open disks,
the approximate carpet $C$ defined by the soup in $D$ will be a subset of the union
$ \cup C_i$ of the approximate carpets defined inside each ball.

For the lower bound on the dimension, consider a sequence $z_n$ of points in $D$ such that
\begin {itemize}
\item The sequence $u_n= d(z_n, \partial D)$ converges to $0$ as $n \to \infty$.
\item The disks $\B (z_n, \sqrt {u_n}), n \ge 0$ are disjoint.
\end {itemize}
It is easy to find such a sequence: Consider a for instance a sequence of points $x_n$ on $\partial D$ and a positive sequence (for instance $v_n= c 2^{-n}$)  in such a way that all the balls
$\B (x_1, v_1)$, \ldots $\B(x_n , v_n)$ etc. are disjoint. Then, one just has to choose for each $n$, a point $z_n$ in $\B (x_n, v_n^2/4) \cap D$.
 
Fix any $R>1$.
Let $F_n$ be the event that there exist curves in $\Gamma_\C$ that intersect both 
$\cir (z_n, Ru_n)$ and $\cir (z_n, \sqrt {u_n})$. 
From Corollary~\ref{thin_large_R} and scale invariance, 
$\PP(F_n) \rightarrow 0$ as $n \rightarrow \infty$,
so by passing to a subsequence we can assume (using a standard
Borel-Cantelli argument) that almost surely, no $F_n$ occurs for large enough $n$.

Now for each $n$, define the soups 
$\Gamma_n = \Gamma_{\B(z_n, 2Ru_n)}$ and  
$\Gamma_n'= \Gamma_{\B(z_n, \sqrt {u_n})}$, and consider the following events:
\begin {itemize}

\item $E_1 (n)$ is the event that for the approximate carpet $C_n$ defined by
the soup $\Gamma_n$, the intersection $C_n \cap \B (z_n, u_n)$ 
has dimension at least $2 - \alpha$. We have showed in the previous section
that it has probability bounded below by some positive constant.

\item $E_2 (n)$ is the event that
there exists a closed loop $\eta_2$ in the ring 
$\{z \ : \ u_n < | z - z_n | < Ru_n \}$ 
that does not cross any curve in $\Gamma_n$.

\item $E_3 (n)$ is the event that there exists no curve 
in $\Gamma_n'$ that intersects both $\cir (z_n,Ru_n)$ and $\cir (z_n, 2Ru_n)$.
\end {itemize}
Since the soup is thin and subcritical, $E_2 (n)$ and $E_3 (n)$ also have probabilities
bounded from below.
The three events are decreasing and therefore positively correlated. 
Hence, the probability of their intersection is bounded below 
independently from $n$. 
Since the disks $\B(z_n , \sqrt {u_n})$ are disjoint, 
it follows that the corresponding soups are independent. 
By Borel-Cantelli, there almost surely exist an infinite set of $n$'s such 
$E_1 (n) \cap E_2 (n) \cap E_3 (n)$ occurs. 

Hence there exists $n$ such that all $E_i(n)$ occur and $F_n$ does not.
Let $z \in C_n \cap \B (z_n, u_n)$, so there is a path $\eta_1$
from $z$ to $\cir (z_n, 2R u_n)$ that does not cross $\Gamma_n$.
The loop $\eta_2$ cannot lie entirely inside $D$, so we can concatenate
parts of $\eta_1$ and $\eta_2$
to construct a path $\eta$ that connects $z$ to $\partial D$, lies
inside $\B(z_n, R u_n)$, and crosses no curve in $\Gamma_n$.
Then $E_3(n)$ guarantees $\eta$ crosses no curve in $\Gamma_n'$,
and finally, since $F_n$ does not occur, $\eta$ crosses no curve in $\Gamma$.
Hence $z \in C$, so the approximate carpet has almost surely
dimension at least
$2 - \alpha$.

We have therefore completed the proof of the following fact:

\begin {proposition}
Under our subcriticality assumption, 
the dimension of the approximate carpet $C$ defined by 
$\Gamma_D$ is almost surely equal to $2- \alpha$. 
\end {proposition}

Finally, we consider the dimension of the carpet $G$ itself:

\begin{Proposition}
Under our subcriticality assumption, 
the dimension of the carpet $G$ defined by 
$\Gamma_D$ is almost surely equal to $2- \alpha$. 
\end{Proposition}

\begin{proof}
Recall that $C \setminus K \subset G \subset C$, where $K$ is the union of all the loops in the loop-soup.
Recall also that subcriticality implies that the soup is thin, which ensures that for any given point $z$, 
$P ( z \in K ) = 0$.

For each $\eps > 0$, we can define the set $C_\eps' = C_\eps \setminus K$. 
Note that  $C \setminus K = \cap_{\eps > 0 } (C_\eps \setminus K ) = \cap_{\eps > 0} C_\eps'$. 
Furthermore, for any given $z$, $P ( z \in C_\eps) = P ( z \in C_\eps')$. 
It follows easily that all first and second moment estimates that we derived for $C_\eps$ 
also hold for $C_\eps'$, so that $\dim (C \setminus K) = 2-\alpha$ with positive probability.
The proof of the 0-1 law also holds essentially unchanged
and we conclude that
the Hausdorff dimension of $G$ is almost surely equal to $2- \alpha$.
\end{proof}

\section {Approximating low-density carpets}

\subsection {The ``remaining set''} 

\label {remaining}

Consider a subcritical thin loop-soup in a bounded domain $D$ as before 
with intensity $\mu$. Recall $O_j$ are the interiors of the loops in the soup,
and define $F= D \setminus \cup_j O_j$. 
$F$ is what we informally call ``the remaining set''. 
We emphasize that {\it a priori} (and in reality) its dimension  
should be larger than that of the approximate carpet. Indeed, a typical 
point in $F$ will be surrounded by infinitely many chains of loops,
and therefore not in the approximate carpet.

\begin{lemma}\label{coeff_c}\cite {Th}
Recall the definition $\beta = \beta (\mu)= \beta (\pi) = \pi( {\cal A}(\gamma) )$,
the ``expected'' area surrounded by $\gamma$. Then 
$\dim(F) = \max (0, 2 - \beta)$.
\end{lemma}

This result was proved in John Thacker's Ph.D. thesis \cite {Th} 
in the context of the Brownian loop-soup.
Since the general proof is essentially identical,
we only give an outline. The proof
is a direct application of the second moment method, in the same spirit as before. 
Things are in fact simpler here, since the loops do not interact, i.e. 
a point $x$ is in the remaining set if and only no loop in the 
Poisson point process belongs to the set of loops $R(x)$ that do not surround $x$.

\begin {proof}
We define $F_\eps$ to be the set obtained by removing from $D$ 
only the interior of the set of loops of diameter greater than $\eps$ 
in the soup, and we apply the second moment method 
(note that $F = \cap_\eps F_\eps$). 
For a given $z$ at positive distance from the boundary of $D$, 
the probability that $z \in F_\eps$ is equal to the probability 
that no loop with diameter greater than $\eps$ in the loop-soup has $z$ 
in its interior i.e. to 
$$ \exp ( - \mu_D ( \{ \gamma \ : \ z \in O( \gamma), |\gamma| > \eps \} )).$$
Using Lemma~\ref{l1}, it follows easily
that, up to multiplicative constants that depend on $|D|$ and
on the distance between $z$ and $\partial D$,
the probability that $z \in F_\eps$ is comparable to 
$\eps^{\beta}$. This is the first-moment estimate.

To bound the second moment, we take two points $x$ and $y$, 
we define $r = d (x,y)/2$ and decompose the loop-soup into three pieces: 
those loops that remain at distance less than $r$ of $x$, 
those loops that remain at distance less than $r$ of $y$, 
and those loops that never come closer to $2r$ of the midpoint between $x$ and $y$.
 We then use the previous argument to deduce that the probability 
that both $x$ and $y$ are in $F_\eps$ is no larger than a constant 
times $(\eps/r)^\beta \times (\eps/r)^\beta \times r^\beta$.  

A $0-1$ type argument analogous to the one that we used for the approximate carpet 
completes the proof.
\end {proof}

\subsection {Varying the intensity of a random soup} \label{sec_intensity}

Suppose that $\mu$ is fixed measure as before, defined from a finite measure $\pi$ as in Lemma \ref {l1}. We are now going to introduce a positive real parameter $c$, and consider for each value of $c$ a soup with intensity $c\mu$ (and its approximate carpet if it is a subcritical soup). 

We can couple realizations of the soups for all $c$ in an increasing manner,
with richer soups for larger $c$.
To see this, one can for instance first define a Poisson point process 
$((\gamma_j, t_j), j \in J)$ with intensity 
$\mu \otimes dt$ on ${\mathcal U} \times [0,t ]$, and then for each $c$, define  
$$ \Gamma_c = ( \gamma_j, j \in J_c ) \hbox { where } J_c = \{ j \in J \ : \ t_j \le c \}$$
and note that $\Gamma_c$ is a Poisson point process with intensity $c \mu$. 
 
For each value of $c$ such that the approximate carpet $C(c)$ is non-empty, its fractal dimension $d(c)$ is almost surely constant (the function $d$ of course depends on the actual choice of $\mu$). It is easy to see 
(this follows for instance from our estimates on the remaining set)
that when $c$ is large, then the approximate carpet is almost surely empty.

On the other hand, comparing the soup with a deterministic well-studied fractal percolation model (see e.g. \cite {Wcras, Wln}), one can show that when $c$ is very small, the approximate carpet is almost surely not empty. Hence, there exists a finite positive critical value $c_0$ that separates these two regimes (this is the origin of the ``subcritical'' terminology). It is in fact not difficult (at least when the measure $\pi$ is invariant under some rotations) to adapt the arguments developed for fractal percolation to prove that the approximate carpet is not empty when $c= c_0$ \cite {Wln}. 

The goal of this section is to  derive first-order estimates for the exponent $d(c)$ when $c \to 0+$.
 Intuitively, in this limit a loop in the soup will typically not intersect any other loop of comparable size. Therefore, the holes in the approximate carpet will look like the interiors of the loops themselves, and this will lead to an approximation of 
its dimension
in terms of the mean area of the interior of loops under the measure $\pi$. The approximate carpet will be rather close to the remaining set.
We will prove:
\begin {proposition}
\label {eq}
Let $\delta (c)$ be the dimension of the remaining set. 
When $c \to 0+$, $d(c) = \delta (c) + o (c) = 2 - c \beta(\pi) + o(c)$.
\end {proposition}

Note that clearly $d(c) \le \delta (c)$ because the approximate carpet is a subset of the remaining set. 

\subsection {Discovering the clusters one by one}

We now describe different ways to ``progressively'' discover loops and clusters 
in a loop-soup.

Consider the soup in some bounded domain $D$. It contains
countably many curves $\gamma_i$, which can be ordered in decreasing order
of their diameter, so $|\gamma_1| > |\gamma_2| > \ldots$ (it is trivial to check that no two loops can have exactly the same diameter).
This induces
an ordering of the clusters of curves in the soup. We start with
$\gamma_1$ and consider its cluster $K_1 = \clus{\gamma_1}{\Gamma}$,
as defined in section~\ref{carpet_cluster}.
Then we take the smallest $i$ such that $\gamma_i$ is not contained in
$K_1$ and consider its cluster $K_2 = \clus{\gamma_i}{\Gamma}$, and so on. 
This yields an ordering of the loop-clusters. 
Again, it is possible to discover $\gamma_1, \gamma_2, \ldots$ 
progressively using the properties of Poisson point processes.
In particular, the conditional law of 
$\gamma_n, \gamma_{n+1}, \ldots $ given $\gamma_1, \ldots , \gamma_{n-1}$ 
is simply that of a Poisson point process of loops 
(ordered according to their diameter) with intensity 
$\mu_D( d \gamma) 1_{ |\gamma |< |\gamma_{n-1}| }$. 
In fact, we are now going to make a variation of this exploration procedure, 
where $K_1$, $K_2$ etc are discovered one by one.

\medbreak

Conditionally on $\gamma_1$ and $K_1$, the law of all other curves in the soup
(not contained in $K_1$) is the same as the law of a standard soup with
components required (i) to have size smaller
than $\gamma_1$ and (ii) not to cross $K_1$.
Hence the soup clusters admit the following equivalent
description (all the soups involved have intensity shape measure $c\pi$):

\begin{itemize}
\item Generate the soup $\Gamma'$, let $\gamma^1$ be
the largest diameter curve in $\Gamma'$.
\item Generate the soup $\Gamma_1$ (independent of $\Gamma'$),
let $\Gamma_1'$ be the subset of curves in $\Gamma_1$ that 
have diameter smaller than $\gamma^1$, and let 
$K_1 = \clus{\gamma^1}{\Gamma_1'}$ be the cluster of
$\gamma^1$ inside $\Gamma_1' \cup \{\gamma^1\} $.
\item Let $\gamma^2$ be
the largest diameter curve in $\Gamma'$ that does not meet $K_1$
\item Generate the soup $\Gamma_2$ (independent of $\Gamma'$, $\Gamma_1$), 
let $\Gamma_2'$ be the subset of curves in $\Gamma_2$ that 
have diameter smaller than $\gamma^2$ and do not
intersect $K_1$, and let
$K_2 = \clus{\gamma^2}{\Gamma_2'}$ be the cluster of
$\gamma^2$ inside $\Gamma_2' \cup \{\gamma^2\}$
\item and continue inductively.
\end{itemize}
This construction may seem unwieldy, as it requires countably
many new soups to construct one, but it will soon prove to be useful.
To summarize our exploration procedure:

\begin{Proposition}\label{soup_via_many_soups}
Let $\Gamma', \Gamma_1, \Gamma_2, \ldots$ be independent loop soups
with intensity $c \mu_D$ on a bounded domain $D$.
For each $n \ge 1$,
we define recursively a curve $\gamma^n$, a set of curves $\Gamma_n'$,
and a cluster $K_n$ as follows:
\begin{itemize}
\item $\gamma^n$ is the largest diameter curve in $\Gamma'$
that does not intersect $\cup_{i=1}^{n-1} K_i$
\item $\Gamma_n'$ is the set of curves in $\Gamma_n$ that have diameter
smaller than $\gamma^n$ and do not intersect $\cup_{i=1}^{n-1} K_i$
\item $K_n = \clus{\gamma^n}{\Gamma_n'}$ is the cluster of
$\gamma^n$ inside $\Gamma_n' \cup \{\gamma_n\}$
\end{itemize}
Then the sequence $(K_n)_{n\ge 1}$ has the same law as the clusters
of a soup of intensity $c\mu_D$ (ordered in decreasing order of
the largest curve they contain).
\end{Proposition}

Hence we can generate soup clusters of a soup by starting with a soup
$\Gamma$, selecting a subset of its curves,
and ``attaching'' to each certain subsets of independent soups $\Gamma_i$. 
If we attach instead
a larger subset or even the whole soup $\Gamma_i$, 
then this can only decrease the size of the corresponding carpet. 
This is what we will do in the next subsection. 

\subsection {Coupling with a ``soup of overlapping clusters''}

To formalize this idea, let us first
consider the product measure $\pi \otimes P_c$, on pairs $(\gamma, \Gamma)$
of one loop (``sampled'' from $\pi$)
and one loop-soup with intensity $c \mu$ {\bf in the entire plane}. 
We have seen that we can choose the finite  measure $\pi$ 
in such a way that it is supported on the set of loops of diameter $1$ 
contained in the square $[-1,1]^2$.
Let $\Gamma'$ be the set of all curves in $\Gamma$ that 
have diameter {smaller} than $1$ and let 
$\gamma^*$ be the ``filling'' of the cluster of $\gamma \cup \Gamma'$ that contains $\gamma$ (i.e. the closure of the complement of the unbounded connected component of the complement of the cluster). We denote by $\pi_c^*$ the measure under which 
$\gamma^*$ is defined.

\begin {lemma}\label{15}
\begin {itemize}
\item
When $c$ is small enough, then $\pi_c^* ( {\cal A} ( \gamma^*) ) < \infty $. 
We denote this quantity $\beta^* (c) $.
\item When  ${c \to 0+}$,  $\beta^*(c)$ converges to  $\beta( \pi)$.
\item There exist $c_1$ and $k$ such that for all $c < c_1$ and all $x>4$, 
the probability that there exists a cluster in $\Gamma'$ that crosses the ring 
$ \{ z \ : \  4 < |z| < x\}$ is bounded by  $k x^{-4}$.
\end {itemize}
\end {lemma}

\begin {proof}
Consider some $c_0 > 0$ so that the loop-soup $\Gamma (c_0)$ 
with intensity $c_0 \mu$ is subcritical.  We know there exists some 
$R>4$ such that the probability $p$ that 
no cluster in $\Gamma (c_0)$ 
traverses the ring $\{ z \ : \ 4 < |z| < R \}$ 
is strictly positive. 

For any $k$, $\Gamma (c_0)$ can also be constructed as the union of
$k$ independent soups of intensity $c_1= c_0 / k$.
Hence we can find a sufficiently large $k$ so that
the probability that the loop-soup $\Gamma (c_1 )$ 
with intensity $c_1  \mu$ contains a cluster that crosses the ring 
is smaller than $1- p^{1/k} \le R^{-4}$.

Using scale-invariance, it follows that for all $l \ge 1$,
 the probability that $\Gamma (c_1)$ contains a cluster that crosses the ring 
$A_l= \{z \ : \ 4 R^l < |z| < R^{l+1} \}$ is smaller than $R^{-4}$. 
Clearly, the same is true if one looks only at the loop-soup $\Gamma'$ 
consisting only of the loops in $\Gamma$ of diameter smaller than $1$.
But the events that  $\Gamma'$ contains a cluster that crosses the ring $A_l$
for $l=0, 1, 2, \ldots$ are in fact independent: they depend only 
on those loops that intersect each of the rings, 
and these sets of loops are disjoint since all loops 
have diameter smaller than $1$. 
We conclude that the probability that a cluster of $\Gamma'$ crosses the ring 
$\{ z \  : \  4 < |z | < R^{l} \}$ is bounded by $R^{-4l}$. This immediately implies the last item of the lemma.

Since $\pi$ is supported on loops of diameter 1 in $[-1,1]^2$,
for $c \le c_1$ and all $l \ge 1$, we have 
$$ \pi_c^* ( \{ | \gamma^* | > R^l \} ) \le \| \pi \| \times R^{-4l}$$
and therefore $\pi_c^*( {\cal A} ( \gamma^* )) \le  4 \pi_c^* ( |\gamma^*|^2) < \infty$.

Finally, observe that if we couple the realizations of $\Gamma'(c)$ for all $c$, 
then almost surely
$\gamma^* (c)$ converges to the filling of the initial loop $\gamma$ when $c \to 0+$,
as all other loops disappear. 
This follows, for instance, from the fact (proved using the same argument as above)
that for any given ring
$\{ z \ : \  4u <  | z - z_0 | < Ru \}$, 
the probability that a loop-soup cluster of $\Gamma (c)$ 
crosses the ring goes to zero as $c$ goes to $0$. 
We conclude, using monotone convergence, that
$$ 
\lim_{c \to 0+} \pi_c^* ( {\cal A} ( \gamma^* ) ) = 
\pi ( {\cal A}(\gamma)) = \beta ( \pi).
$$
\end {proof}

We now construct the measure $\mu_c^*$ from $\mu$ similarly
to the way $\pi_c^*$ was constructed from $\pi$:
Consider the product measure $c \mu \otimes P_c$, on pairs $(\gamma, \Gamma)$
of one loop (``sampled'' from $c\mu$)
and one loop-soup with intensity $c \mu$ {in the entire plane}. 
Let $\Gamma' ( \gamma)$ be the set of all curves in $\Gamma$ that 
have diameter {smaller} than $\gamma$ and let 
$\gamma^*$ be the ``filling'' of the cluster of $\gamma \cup \Gamma'(\gamma)$ that contains $\gamma$. We denote by $\mu_c^*$ the measure under which 
$\gamma^*$ is defined. Clearly, scale-invariance and translation-invariance of $\mu$ (and of the loop-soup) imply that $\mu_c^*$ is also scale-invariant, translation-invariant, and that it can be constructed (as in section 2) from $\pi_c^*$ (note that the previous lemma ensures that $  \pi^*_c ( |\gamma^*|^2) < \infty$).

In fact, this definition makes it possible to define a Poisson point process of 
pairs $(\gamma_j, \gamma_j^*)$  i.e. to couple a loop-soup with intensity $c\mu$ with a loop-soup with intensity $\mu_c^*$ in such a way that for each $j$, $\gamma_j \subset \gamma_j^*$ (basically, each loop $\gamma_j$ is extended by an independent soup of loops of smaller diameter in the whole plane). 
If we keep those $\gamma_j$ that are in a domain $D$, 
we get a loop-soup $\Gamma_D$ with intensity $c\mu_D$. 
But mind that the corresponding $\gamma_j^*$'s do not necessarily stay in $D$. 
However, they are not likely to be very large, as we now show.
To keep things simple, we assume for the rest of this section that $D$ is
the unit disk; we can do this without loss of generality, since the dimension
of the carpet does not depend on the domain.

\begin {lemma}
Let $c_0$ and $R$ be defined as in the proof of Lemma~\ref{15}, 
and consider the coupling $(\gamma_j, \gamma_j^*)$ defined above.
If $c$ is small enough, then the probability that 
for all $j$ such that $\gamma_j \subset D$, the diameter of $\gamma_j^*$ does not exceed $2R$ 
is strictly positive.
\end {lemma}

\begin {proof}
This is in fact a simple consequence of the last statement of Lemma \ref {15}. It suffices to show that the set 
$$ \{ (\gamma, \gamma^*) \  : \  |\gamma^*| > 2R, \  \gamma \subset D \} $$
has finite mass. 
Note that because of scale-invariance,  the $\mu$-mass of the set of loops in $D$ with diameter between $2^{-l-1}$ and $2^{-l}$ grows 
slower than $O(2^{2l})$ as $l \to \infty$. 
But as $\sum_l 4^l \times (2^{-l}/ 2R)^{-4}$ converges, the lemma follows readily.  
\end {proof}

We are now ready to prove the following lemma that will enable us to conclude the proof 
of Proposition \ref {eq}. 

\begin {lemma}
Let $\Gamma_D$ be a subcritical soup with intensity $c \mu$ in the unit disk, 
and let $C_\eps$ defined as before. For $c$ small enough, 
there exists  a constant $k$ such that for all small $\eps$,
$\PP ( 0 \in C_\eps) \ge 
 k \eps^{c\beta^*(c)}$,
 where $\beta^* (c)=  \pi_c^* ( {\cal A} ( \gamma^*) )$.
\end{lemma}

\begin {proof}
Let us consider the construction of the loop soup clusters $(K_n)$ in $D$ described in Proposition \ref {soup_via_many_soups}. Clearly, if we compare it with the 
coupling that we have just described, we can choose our coupling  $(\gamma_j, \gamma_j^*)$ in such a way that each $K_n$ is in fact a subset of one of the $\gamma_j^*$'s with $\gamma_j \subset D$. 
 One therefore has 
 $$ 
 P ( 0 \in C_\eps ) \ge P ( E') 
 $$
 where $E'$ denotes the event that none of the sets $\gamma_j^*$ with $\gamma_j \subset D$ surround the disc of radius $\eps$ 
(this is because each cluster in the initial loop-soup is contained in some $\gamma_j^*$).
 This last event depends on the Poisson point process of pairs $(\gamma_j, \gamma_j^*)$ only.

Let $E''$ denote the event that
 that for all $j$ such that $\gamma_j \subset D$, one has 
$\gamma_j^* \subset (4R) \cdot D$. Recall from the previous lemma that this event has a positive probability (provided $c$ is small enough).
 Let $E'''$ denote the event that none of the sets $\gamma_j^*$ that are subsets of $4RD$ and of diameter greater than $\eps$ surround the origin.
 This event is independent of $E''$, as they occur for disjoint set of pairs $(\gamma, \gamma^*)$. Note that $E''' \cap E'' \subset E' \cap E''$. Hence, 
 $$ P(E) \ge P ( E' \cap E'') \ge P (E''' \cap E'')  = P( E''') P (E'') .
 $$
 It therefore remains to estimate $P(E''')$. This is equivalent to 
an estimate concerning the remaining set of the Poisson point process of 
 sets $(\gamma^*)$ in $(4R) \cdot D$, and is obtained 
just as for the first-moment estimate in Lemma \ref{coeff_c}.  
 \end {proof}
 
We can now conclude the proof of Proposition \ref {eq}. Recall that $d(c) \le \delta (c ) = 2 - c \beta (\pi)$ because the remaining set contains the carpet. On the other hand, 
the previous lemma and the definition of $\beta (\pi)$ show that the $\alpha = 2 - d(c)$ corresponding to the loop-soup with intensity $c\pi$ is not larger than $c \beta^*(c) = c \beta(\pi) + o (c)$.

\subsection {Consequences for the Brownian loop-soup}

As we have already mentioned, in the case of the Brownian loop-soup, 
the connection between the boundaries of clusters derived in 
\cite {ShW,ShW2} (see also \cite {Wcras}) and the Schramm-Loewner Evolutions 
make it possible to describe the carpet via branching SLE-type paths \cite {Sh}. 
Schramm, Sheffield and Wilson \cite {SchShW} computed in fact the 
``expected dimension'' of the carpet, and our paper shows that this 
``expected dimension'' (the quantity that governs the first moment) is equal to
the almost sure Hausdorff dimension $d(c)$. 
Combining all these results, the expression derived 
in \cite {SchShW} implies that for $c\le 1$,
$$ 
d(c) = 2 - \frac {(3 \kappa - 8)(8 - \kappa)} {32 \kappa}
$$
where $c$ and $\kappa$ are related by the ``usual'' relation 
$
c(\kappa) = ({(3\kappa -8) (6 - \kappa)})/({2 \kappa})$. 
It follows that
$$
d(c) = 2 - \frac {c}{16} - \frac {1}{96} \left( 5 + c - \sqrt { 25 + c^2 - 26 c } \right). 
$$
Note that $d(c) = 2 - c/10 + o (c)$. 
As explained in \cite {Th} (in the context of the remaining set), the coefficient of the first order term $c/10$ is closely related to the mean area of the Brownian loop computed in \cite {GTF}.

Recall \cite {Be, Ladim} that the Hausdorff dimension of the SLE$_\kappa$ curve is almost surely equal to  $1 + \kappa/8$. Hence, in the Brownian loop-soup case, we also know the dimension of the boundary of the carpet.


D\'epartement de Math\'ematiques et Applications 

Ecole Normale Sup\'erieure 

45, rue d'Ulm

75230 Paris cedex 05 France






\medbreak


Laboratoire de Math\'ematiques 

B\^at. 425, Universit\'e Paris-Sud 

91405 Orsay cedex, France

\medbreak

serban.nacu@gmail.com

wendelin.werner@math.u-psud.fr

\end{document}